\documentclass[4pt]{article} 

\usepackage[utf8]{inputenc} 
\usepackage{geometry} 
\usepackage{graphicx} 

\usepackage[colorlinks,citecolor=red]{hyperref}
\usepackage{amsmath}
\usepackage{amsfonts}
\usepackage{dsfont}
\usepackage{mathrsfs}
\usepackage{latexsym}
\usepackage{graphicx}
\usepackage{amssymb}
\usepackage{float}
\usepackage{epsfig}
\usepackage{epstopdf}
\usepackage{color,xcolor}
\usepackage{booktabs} 
\usepackage{array} 
\usepackage{paralist} 
\usepackage{verbatim} 
\usepackage{subfig} 

\newtheorem{theorem}{Theorem}[section]
\newtheorem{lemma}[theorem]{Lemma}
\newtheorem{corollary}[theorem]{Corollary}

\newtheorem{remark}[theorem]{Remark}

\usepackage{fancyhdr} 
\usepackage{sectsty}
\allsectionsfont{\sffamily\mdseries\upshape} 

\usepackage[nottoc,notlof,notlot]{tocbibind} 
\usepackage[titles,subfigure]{tocloft} 

\title{Vertex-connectivity and $Q$-index of graphs with fixed girth}

\author{Huicai Jia \thanks{School of Mathematics, Renmin University of China, Beijing, 100872, China; College of Science, Henan
University of Engineering, Zhengzhou, Henan 451191, China. Email: jhc607@163.com},
Hong-Jian Lai \thanks{Department of Mathematics, West Virginia
University, Morgantown, WV 26506, USA. Email: hjlai@math.wvu.edu},
Ruifang Liu \thanks{Corresponding author. School of Mathematics and Statistics, Zhengzhou
University, Zhengzhou, Henan 450001, China. Email: rfliu@zzu.edu.cn},
Ju Zhou \thanks{Department of Mathematics, Kutztown University of Pennsylvania, Kutztown, PA 19530, USA. Email: zhou@kutztown.edu}}

\date{ }

\begin{document}

\maketitle

\begin{abstract}
Let $q(G)$ denote the $Q$-index of a graph $G$, which is the largest signless Laplacian eigenvalue of $G$.
We prove best possible upper bounds of $q(G)$ and best possible lower bounds of $q(\overline{G})$ for
a connected graph $G$ to be $k$-connected and maximally connected, respectively. Similar upper bounds of $q(G)$ and lower bounds of $q(\overline{G})$ to
assure $G$ to be super-connected are also obtained. Upper bounds of $q(G)$ and lower bounds of $q(\overline{G})$ to
assure a connected triangle-free graph $G$ to be $k$-connected, maximally connected and super-connected are also respectively investigated.
\end{abstract}



\section{Introduction}
We consider simple, undirected and connected graphs. Let $G$ be a
simple graph with vertex set $V(G)$ and edge set
$E(G)$ such that $|V(G)|=n$ and $|E(G)|=m$. Thus $G$ can be viewed as a spanning subgraph of $K_n$.
Define the {\bf complement} of $G$ to be the graph $\overline{G} = K_n - E(G)$.
We denote by $d(v)$ and
$\delta(G)$ the degree of a vertex $v$ in $G$ and the minimum degree of $G$, respectively.
Let $K_{n}$ and $K_{a, b}$ denote complete graphs and complete bipartite graphs
on $n$ vertices, where $a+b=n.$ For two disjoint subsets $X$ and $Y$ of $V(G)$,
let $E(X, Y)$ be the set of edges with one end in $X$ and the other end in $Y$.
The join of $G$ and $H$, denoted by $G\vee H$, is the graph obtained from a disjoint
union of $G$ and $H$ by adding all possible edges between them.
Let $G\cup H$ and $G[V_{0}]$~$(V_{0}\subseteq V(G))$ denote the disjoint union of $G$
and $H$ and the subgraph of $G$ induced by $V_{0}$, respectively.
Assume $e=uv\in E(G)$, let $G-e$ be the subgraph of $G$ by deleting $e$ from $G$.
Let $G-V_{0}$ be the induced subgraph obtained from $G$ by deleting
the vertices of $V_{0}$ together with the edges. The girth of a graph $G$, is defined as
\[
g(G) = \left\{
\begin{array}{ll}
\min\{|E(C)|: \mbox{ $C$ is a cycle of $G\}$ } & \mbox{ if $G$ is not acyclic, }
\\
\infty & \mbox{ if $G$ is acyclic. }
\end{array} \right.
\]

A vertex subset $C$ of a connected graph $G$ is called a
vertex-cut if $G-C$ is not connected or $G-C = K_1$. The vertex connectivity $\kappa(G)$ of
a connected non-complete graph $G$ is the minimum number of vertices whose deletion disconnects $G$.
A vertex-cut $C$ is minimum if $|C|=\kappa(G)$. A well-known result of Whitney \cite{Whit32}
states that $\kappa(G)\leq \delta(G)$ for any graph $G$. A graph $G$ is {\bf $k$-connected} if
$\kappa(G)\geq k$, {\bf maximally connected} if $\kappa(G)=\delta(G)$, and
{\bf super-$\kappa$} (or {\bf super-connected}) if each minimum vertex-cut isolates a vertex of
minimum degree. Hence every super-$\kappa$ graph must be maximally connected. A triangle-free
graph is an undirected graph with no induced 3-cycle.
We follow Bondy and Murty \cite{B} for notation and terminologies not defined here.

The adjacency matrix of $G$ is the
$n\times n$ matrix $A(G)=(a_{ij})$, where $a_{ij}=1$ if $v_i$ and $v_j$ are adjacent and
otherwise $a_{ij}=0$. Let $D(G)$ be the diagonal matrix of the vertex degrees of $G$.
The matrix $Q(G) = D(G) + A(G)$ is known as the {\bf $Q$-matrix} or the {\bf signless
Laplacian} matrix  of $G$. We denote the largest eigenvalue of $Q(G)$ by $q(G)$,
which is called {\bf $Q$-index}  or the {\bf signless Laplacian spectral radius} of $G$.

There have been quite a few recent studies on the relationship between vertex-connectivity
and eigenvalues of graphs. O \cite{O} presented the relation between  vertex-connectivity
and the second largest eigenvalue of regular multigraphs. Abiad et al. \cite{A} proved upper
bounds for the second largest eigenvalues of regular graphs and multigraphs which guarantee a desired vertex-connectivity.
Recently, Liu et al. \cite{Liu} investigated functions $f(\delta, \Delta, g, k)$ with
$\Delta\geq\delta \ge k \ge 2$ and girth $g\geq3$ such that any graph $G$ satisfying $\lambda_2(G) < f(\delta, \Delta, g, k)$ has connectivity $\kappa(G)\ge k$.
On the other hand, Li \cite{Li} presented
sufficient conditions for a graph to be $k$-connected in terms of the spectral radius and $Q$-index.
Hong et al. \cite{HXC} found sufficient conditions for a connected graph and a
connected triangle-free graph with given minimum degree to be $k$-connected, maximally connected and super-connected
in terms of the spectral radius of the graph and of its complement, respectively.
Zhang et al. \cite{FZL} proved a sufficient condition for a connected graph with fixed
minimum degree to be $k$-connected based on $Q$-index for sufficiently large order $n$.

Motivated by these results, the purpose of the current research focuses on the following general problem.

\noindent{\it{\bf Problem 1.1} For a connected graph $G$ with fixed girth $g\geq3$ and minimum degree $\delta\geq k\geq2$, find optimal sufficient conditions in terms of $Q$-index of the graph and of its complement to describe the properties of being $k$-connected, maximally connected and
super-connected.}

In particular, we in this paper investigate the problem above for the two special cases: connected graphs ($g\geq3$) and connected triangle-free graphs ($g\geq4$). In the next section, we display some useful tools to be deployed in our arguments. In the subsequent sections, our main results for the generic study and for the special cases are presented and justified.

\section{Preliminaries}

We in this section will present some former results that will be utilized in our arguments.
The following bounds of the $Q$-index of a graph $G$,
stated in Lemmas \ref{le2.1} and \ref{le2.5}, are applied frequently.

\begin{lemma}(Cvetkovi\'{c}, Rowlinson and Simi\'{c} \cite{CRS})\label{le2.1}
Let $G$ be a graph with order $n$ and size $m.$ Then
$$q(G)\geq \frac{4m}{n}.$$
If $G$ is connected, then the equality holds if and only if $G$ is a regular graph.
\end{lemma}

\begin{lemma}(Feng and Yu \cite{FY})\label{le2.5}
Let $G$ be a connected graph with $n$ vertices and $m$ edges. Then
$$q(G)\leq \frac{2m}{n-1}+n-2,$$ and the equality holds if and only if $G$ is $K_{n}$ or $K_{1,n-1}.$
\end{lemma}

Given positive integers $\delta$, $g$ and $\kappa$, define $\displaystyle  t = \lfloor \frac{g-1}{2} \rfloor$ and
\begin{equation*} \label{nu0}
\nu(\delta, g, \kappa) =
\left\{
\begin{array}{ll}
1+(\delta-\kappa)\sum_{i=0}^{t-1}(\delta-1)^{i} &\,~~~ \text{\mbox{if}~ $g=2t+1$},\\
2+(2\delta-2-\kappa)\sum_{i=0}^{t-1}(\delta-1)^{i} &\, ~~~\text{\mbox{if}~ $g=2t+2$ \mbox{and} $\delta\geq3$},\\
2t+1 &\, ~~~\text{\mbox{if}~ $g=2t+2$ \mbox{and} $\delta=2$}.
\end{array}
\right.
\end{equation*}

Liu et al. \cite{Liu} proved the following result which is crucial to our main results in Section 3.

\begin{lemma}(Liu, Lai, Tian and Wu \cite{Liu})\label{le2.10}
Let $G$ be a simple connected graph with $\kappa(G)=\kappa$, minimum degree $\delta\geq k\ge2$ and girth $g\ge3.$ Let $C$ be a minimum vertex cut of $G$ with $|C|=\kappa$ and $V_{0}$ be a connected component of $G-C$. If $\kappa\leq k-1<\delta$, then $$|V_{0}|\geq \nu(\delta, g, \kappa)\geq\nu(\delta, g, k-1).$$
\end{lemma}

In \cite{FS}, F\"{u}redi et al. proved the following girth and Tur\'{a}n number result.

\begin{lemma}(F\"{u}redi and Simonovits \cite{FS})\label{le2.11}
Let $G$ be a simple connected graph with order $n$, size $m$ and girth $g\ge3.$ Then
\begin{equation*} \label{nu0}
m<
\left\{
\begin{array}{ll}
\frac{1}{2}n^{1+\frac{1}{t}}+\frac{1}{2}n &\,~~~  \text{\mbox{if}~ $g=2t+1$},\\
\frac{1}{2^{1+\frac{1}{t}}}n^{1+\frac{1}{t}}+\frac{1}{2}n &\, ~~~\text{\mbox{if}~ $g=2t+2$}.
\end{array}
\right.
\end{equation*}
\end{lemma}

\begin{lemma}\label{le2.2}
Let $G=K_{\kappa}\vee(K_{a}\cup K_{b}),$ where $\delta-\kappa+1\leq a\leq b\leq n-\delta-1$ and $a+b=n-\kappa$,  then
$q(G)>n-2.$
\end{lemma}

\begin{proof}
Since $\delta-\kappa+1\leq a\leq b\leq n-\delta-1$ and $a+b=n-\kappa$, we have $a\leq \frac{n-\kappa}{2}\leq b$.
By Lemma \ref{le2.1}, then
\begin{eqnarray*}
q(G)&\geq& \frac{4m}{n} = \frac{4}{n}[\frac{n(n-1)}{2}-ab] = 2(n-1)-\frac{4}{n}a(n-\kappa-a)\\
    &=&2(n-1)+\frac{4}{n}[(a-\frac{n-\kappa}{2})^{2}-\frac{(n-\kappa)^{2}}{4}]
     \geq  2(n-1)-\frac{4}{n}\frac{(n-\kappa)^{2}}{4}\\
    &\geq& 2(n-1)-\frac{4}{n}\frac{(n-1)^{2}}{4} > n-2.
\end{eqnarray*}
The result follows.\hspace*{\fill}$\Box$
\end{proof}

The following Tur\'{a}n's Theorem is well known.

\begin{theorem}(Mantel \cite{MW} and Tur\'{a}n \cite{TP})\label{le2.13}
For any triangle-free graph $G$ of order $n$ and size $m$, we have
$$m\leq \lfloor\frac{1}{4}n^{2}\rfloor,$$ with equality if and only if
$G\cong K_{\lfloor\frac{n}{2}\rfloor, \lceil\frac{n}{2}\rceil}$.
\end{theorem}

\section{Vertex-connectivity and $Q$-index of graphs with fixed girth}
Motivated by the methods deployed in \cite{HXC, Liu}, in this section, we mainly give
sufficient conditions on $q(G)$ and $q(\overline{G})$ to predict a connected
graph $G$ with fixed girth $g$ to be $k$-connected.

First, we present a crucial and technical lemma.
\begin{lemma}\label{le3.1}
Let $G$ be a connected graph of order $n$, size $m$, minimum degree $\delta\geq k\geq2$ and girth $g\geq3.$
Define $\nu=\nu(\delta, g, k-1)$. If
\begin{equation} \label{nu-0}
m(G) \geq
\left\{
\begin{array}{ll}
\frac{1}{2}(\nu+k-1)^{1+\frac{1}{t}}+\frac{1}{2}(n-\nu)^{1+\frac{1}{t}}+\frac{1}{2}(n+k-1) &\,~~~ \text{\mbox{if}~ $g=2t+1$},\\
\frac{1}{2^{1+\frac{1}{t}}}(\nu+k-1)^{1+\frac{1}{t}}+\frac{1}{2^{1+\frac{1}{t}}}(n-\nu)^{1+\frac{1}{t}}+\frac{1}{2}(n+k-1) &\, ~~~\text{\mbox{if}~ $g=2t+2$},
\end{array}
\right.
\end{equation}
then $G$ is $k$-connected.
\end{lemma}

\begin{proof}
Assume that $\kappa\leq k-1$. Let $C$ be a minimum vertex-cut of $G$, then $|C|=\kappa\leq k-1<\delta.$ Let $V_{0}, V_{1}, \ldots, V_{t-1}$~$(t\geq2)$ be the vertex sets of connected components of $G-C$ with $|V_{0}|\leq|V_{1}|\leq \cdots \leq|V_{t-1}|,$
and let $U=\bigcup_{i=1}^{t-1}V_{i}$~(see Figure 1).
\setlength{\unitlength}{0.8pt}
\begin{center}
\begin{picture}(500,200)
\qbezier(212,90)(212,115)(221,133)\qbezier(221,133)(230,152)(244,152)
\qbezier(244,152)(257,152)(266,133)\qbezier(266,133)(276,115)(276,90)
\qbezier(276,90)(276,64)(266,46)\qbezier(266,46)(257,28)(244,28)
\qbezier(244,28)(230,28)(221,46)\qbezier(221,46)(212,64)(212,90)
\qbezier(330,89)(330,114)(339,132)\qbezier(339,132)(348,151)(362,151)
\qbezier(362,151)(375,151)(384,132)\qbezier(384,132)(394,114)(394,89)
\qbezier(394,89)(394,63)(384,45)\qbezier(384,45)(375,27)(362,27)
\qbezier(362,27)(348,27)(339,45)\qbezier(339,45)(330,63)(330,89)
\put(146,109){\circle*{4}}
\put(224,109){\circle*{4}}
\qbezier(146,109)(185,109)(224,109)
\put(147,88){\circle*{4}}
\put(225,88){\circle*{4}}
\qbezier(147,88)(186,88)(225,88)
\put(147,66){\circle*{4}}
\put(225,66){\circle*{4}}
\qbezier(147,66)(186,66)(225,66)
\put(265,116){\circle*{4}}
\put(343,116){\circle*{4}}
\qbezier(265,116)(304,116)(343,116)
\put(265,96){\circle*{4}}
\put(343,96){\circle*{4}}
\qbezier(265,96)(304,96)(343,96)
\put(265,77){\circle*{4}}
\put(343,77){\circle*{4}}
\qbezier(265,77)(304,77)(343,77)
\put(265,61){\circle*{4}}
\put(343,61){\circle*{4}}
\qbezier(265,61)(304,61)(343,61)
\put(124,80){$V_{0}$}
\put(241,80){$C$}
\put(362,80){$U$}
\qbezier(160,90)(160,115)(152,133)
\qbezier(152,133)(144,152)(133,152)\qbezier(133,152)(121,152)(113,133)
\qbezier(113,133)(106,115)(105,90)
\qbezier(105,90)(106,64)(113,46)\qbezier(113,46)(121,28)(133,28)
\qbezier(133,28)(144,28)(152,46)
\qbezier(152,46)(160,64)(160,90)
\end{picture}
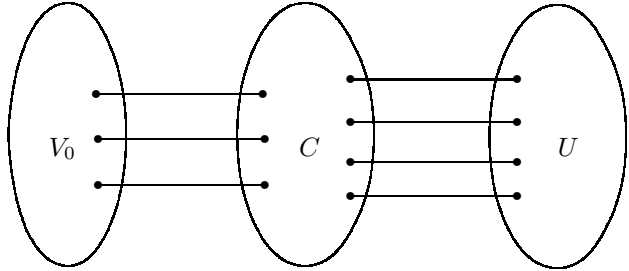
\vskip 0.1cm Figure $1$. The partition of $V(G)$ into $V_{0}$, $C$ and $U.$
\end{center}
By Lemma \ref{le2.10}, for any $i$ with $0\leq i\leq t-1$, we have
$$|V_{i}|\geq \nu(\delta, g, \kappa).$$
In particular, $\nu(\delta, g, \kappa)\leq |V_{0}|\leq |U|\leq n-\kappa-\nu(\delta, g, \kappa)$ and $|V_{0}|+|U|=n-\kappa$.
In the following, we proceed our proof according to the different parities of the girth $g$.

\noindent {\bf Case 1.} $g=2t+1$ is odd.

By Lemma \ref{le2.11}, and since $\nu(\delta, g, \kappa)\leq|V_{0}|\leq \frac{n-\kappa}{2}\leq|U|,$ we have
\begin{eqnarray*}
  m(G) &=& |E(G[V_{0}\cup C])|+|E(G[C\cup U])|-|E(G[C])|\\
       &\leq& |E(G[V_{0}\cup C])|+|E(G[C\cup U])|\\
      &<& \frac{1}{2}(|V_{0}|+|C|)^{1+\frac{1}{t}}+\frac{1}{2}(|V_{0}|+|C|)+\frac{1}{2}(|C|+|U|)^{1+\frac{1}{t}}+\frac{1}{2}(|C|+|U|)\\
      &=& \frac{1}{2}(|V_{0}|+\kappa)^{1+\frac{1}{t}}+\frac{1}{2}(n-|V_{0}|)^{1+\frac{1}{t}}+\frac{1}{2}(n+\kappa)~~~(\text{\mbox{decreasing~on}~} |V_{0}|)\\
       &\leq& \frac{1}{2}(\nu(\delta, g, \kappa)+\kappa)^{1+\frac{1}{t}}+\frac{1}{2}(n-\nu(\delta, g, \kappa))^{1+\frac{1}{t}}+\frac{1}{2}(n+\kappa)~~~(\text{\mbox{increasing~on}~} \kappa)\\
       &\leq& \frac{1}{2}(\nu+k-1)^{1+\frac{1}{t}}+\frac{1}{2}(n-\nu)^{1+\frac{1}{t}}+\frac{1}{2}(n+k-1),
\end{eqnarray*}
contrary to (\ref{nu-0}) and so Case 1 is justified.

\noindent {\bf Case 2.} $g=2t+2$ is even.

By Lemma \ref{le2.11}, with a similar argument as in Case 1, we obtain
\begin{eqnarray*}
  m(G) &=& |E(G[V_{0}\cup C])|+|E(G[C\cup U])|-|E(G[C])|\\
      &<& \frac{1}{2^{1+\frac{1}{t}}}(|V_{0}|+|C|)^{1+\frac{1}{t}}+\frac{1}{2}(|V_{0}|+|C|)+\frac{1}{2^{1+\frac{1}{t}}}(|C|+|U|)^{1+\frac{1}{t}}+\frac{1}{2}(|C|+|U|)\\
      &=& \frac{1}{2^{1+\frac{1}{t}}}(|V_{0}|+\kappa)^{1+\frac{1}{t}}+\frac{1}{2^{1+\frac{1}{t}}}(n-|V_{0}|)^{1+\frac{1}{t}}+\frac{1}{2}(n+\kappa)\\
       &\leq& \frac{1}{2^{1+\frac{1}{t}}}(\nu(\delta, g, \kappa)+\kappa)^{1+\frac{1}{t}}+\frac{1}{2^{1+\frac{1}{t}}}(n-\nu(\delta, g, \kappa))^{1+\frac{1}{t}}+\frac{1}{2}(n+\kappa)\\
       &\leq& \frac{1}{2^{1+\frac{1}{t}}}(\nu+k-1)^{1+\frac{1}{t}}+\frac{1}{2^{1+\frac{1}{t}}}(n-\nu)^{1+\frac{1}{t}}+\frac{1}{2}(n+k-1),
\end{eqnarray*}
contrary to (\ref{nu-0}) and so Case 2 is justified. This completes the proof of the lemma. \hspace*{\fill}$\Box$
\end{proof}

\begin{theorem}\label{th3.2}
Let $G$ be a connected graph of order $n$, minimum degree $\delta \geq k \geq2$ and girth $g\geq3$,
and let $\nu=\nu(\delta, g, k-1)$.
If
\begin{equation*} \label{nu0}
q(G) \geq
\left\{
\begin{array}{ll}
\frac{(\nu+k-1)^{1+\frac{1}{t}}+(n-\nu)^{1+\frac{1}{t}}}{n-1}+\frac{k}{n-1}+(n-1) &\,~~~ \text{\mbox{if}~ $g=2t+1$},\\
\frac{(\nu+k-1)^{1+\frac{1}{t}}+(n-\nu)^{1+\frac{1}{t}}}{2^{\frac{1}{t}}(n-1)}+\frac{k}{n-1}+(n-1) &\, ~~~\text{\mbox{if}~ $g=2t+2$},
\end{array}
\right.
\end{equation*}
then $G$ is $k$-connected.
\end{theorem}

\begin{proof}
By Lemma \ref{le2.5}, we have
\begin{equation*} \label{nu0}
\frac{2m}{n-1}+n-2 \geq q(G) \geq
\left\{
\begin{array}{ll}
\frac{(\nu+k-1)^{1+\frac{1}{t}}+(n-\nu)^{1+\frac{1}{t}}}{n-1}+\frac{k}{n-1}+(n-1) &\,~~~ \text{\mbox{if}~ $g=2t+1$},\\
\frac{(\nu+k-1)^{1+\frac{1}{t}}+(n-\nu)^{1+\frac{1}{t}}}{2^{\frac{1}{t}}(n-1)}+\frac{k}{n-1}+(n-1) &\, ~~~\text{\mbox{if}~ $g=2t+2$}.
\end{array}
\right.
\end{equation*}
Then
\begin{equation*} \label{nu0}
m \geq
\left\{
\begin{array}{ll}
\frac{1}{2}(\nu+k-1)^{1+\frac{1}{t}}+\frac{1}{2}(n-\nu)^{1+\frac{1}{t}}+\frac{1}{2}(n+k-1) &\,~~~ \text{\mbox{if}~ $g=2t+1$},\\
\frac{1}{2^{1+\frac{1}{t}}}(\nu+k-1)^{1+\frac{1}{t}}+\frac{1}{2^{1+\frac{1}{t}}}(n-\nu)^{1+\frac{1}{t}}+\frac{1}{2}(n+k-1) &\, ~~~\text{\mbox{if}~ $g=2t+2$}.
\end{array}
\right.
\end{equation*}
By Lemma \ref{le3.1}, $G$ is $k$-connected.
\hspace*{\fill}$\Box$
\end{proof}

\begin{theorem}\label{th3.3}
Let $G$ be a connected graph of order $n$, minimum degree $\delta \geq k \geq2$ and girth $g\geq3$,
and let $\nu=\nu(\delta, g, k-1)$.
If
\begin{equation} \label{nu-1}
q(\overline{G})\leq
\left\{
\begin{array}{ll}
2(n-1)-\frac{2}{n}(\nu+k-1)^{1+\frac{1}{t}}-\frac{2}{n}(n-\nu)^{1+\frac{1}{t}}-\frac{2}{n}(n+k-1)&\,~~~ \text{\mbox{if}~ $g=2t+1$},\\
2(n-1)-\frac{2}{n\cdot2^{\frac{1}{t}}}(\nu+k-1)^{1+\frac{1}{t}}-\frac{2}{n\cdot2^{\frac{1}{t}}}(n-\nu)^{1+\frac{1}{t}}-\frac{2}{n}(n+k-1) &\, ~~~\text{\mbox{if}~ $g=2t+2$},
\end{array}
\right.
\end{equation}
then $G$ is $k$-connected.
\end{theorem}

\begin{proof}
By contradiction, we assume that $\kappa(G) \leq k-1$.
We argue according to the different parities of the girth $g$.

\noindent {\bf Case 1.} $g=2t+1$ is odd.

As $\kappa(G) \leq k-1$, by Lemma \ref{le3.1},
we have  $m(G)<\frac{1}{2}(\nu+k-1)^{1+\frac{1}{t}}+\frac{1}{2}(n-\nu)^{1+\frac{1}{t}}+\frac{1}{2}(n+k-1)$.
It follows from $m(G)+m(\overline{G})=\frac{n(n-1)}{2}$ that
\begin{eqnarray*}
m(\overline{G})&=& \frac{n(n-1)}{2}-m(G)\\
              &>& \frac{n(n-1)}{2}-\frac{1}{2}(\nu+k-1)^{1+\frac{1}{t}}-\frac{1}{2}(n-\nu)^{1+\frac{1}{t}}-\frac{1}{2}(n+k-1).
\end{eqnarray*}
By Lemma \ref{le2.1}, a contradiction to (\ref{nu-1}) is obtained.
\begin{eqnarray*}
q(\overline{G})&\geq& \frac{4m(\overline{G})}{n}\\
              &>& \frac{4}{n}[\frac{n(n-1)}{2}-\frac{1}{2}(\nu+k-1)^{1+\frac{1}{t}}-\frac{1}{2}(n-\nu)^{1+\frac{1}{t}}-\frac{1}{2}(n+k-1)]\\
                &=& 2(n-1)-\frac{2}{n}(\nu+k-1)^{1+\frac{1}{t}}-\frac{2}{n}(n-\nu)^{1+\frac{1}{t}}-\frac{2}{n}(n+k-1).
\end{eqnarray*}

\noindent {\bf Case 2.} $g=2t+2$ is even.

As $\kappa(G) \leq k-1$, by Lemma \ref{le3.1},
we have   $m(G)<\frac{1}{2^{1+\frac{1}{t}}}(\nu+k-1)^{1+\frac{1}{t}}+\frac{1}{2^{1+\frac{1}{t}}}(n-\nu)^{1+\frac{1}{t}}+\frac{1}{2}(n+k-1)$.
Since $m(G)+m(\overline{G})=\frac{n(n-1)}{2},$ we have
\begin{eqnarray*}
m(\overline{G})&=& \frac{n(n-1)}{2}-m(G)\\
              &>& \frac{n(n-1)}{2}-\frac{1}{2^{1+\frac{1}{t}}}(\nu+k-1)^{1+\frac{1}{t}}-\frac{1}{2^{1+\frac{1}{t}}}(n-\nu)^{1+\frac{1}{t}}-\frac{1}{2}(n+k-1).
\end{eqnarray*}
By Lemma \ref{le2.1}, we obtain a contradiction to (\ref{nu-1}) again.
\begin{eqnarray*}
q(\overline{G})&\geq& \frac{4m(\overline{G})}{n}\\
              &>& \frac{4}{n}[\frac{n(n-1)}{2}-\frac{1}{2^{1+\frac{1}{t}}}(\nu+k-1)^{1+\frac{1}{t}}-\frac{1}{2^{1+\frac{1}{t}}}(n-\nu)^{1+\frac{1}{t}}-\frac{1}{2}(n+k-1)]\\
                &=& 2(n-1)-\frac{2}{n\cdot2^{\frac{1}{t}}}(\nu+k-1)^{1+\frac{1}{t}}-\frac{2}{n\cdot2^{\frac{1}{t}}}(n-\nu)^{1+\frac{1}{t}}-\frac{2}{n}(n+k-1).
\end{eqnarray*}
These contradictions establish Theorem \ref{th3.3}.
\end{proof}
\hspace*{\fill}$\Box$

\begin{remark}
In fact, by taking $k=\delta$ in Lemma \ref{le3.1} and $\kappa=\delta$ in the proof of Lemma \ref{le3.1}, we can prove sufficient conditions on size $m$
for a connected graph with fixed girth to be maximally connected and super-connected, respectively. Using sufficient conditions on size $m$,
we can also obtain sufficient conditions on $q(G)$ and $q(\overline{G})$ to ensure a connected graph with fixed
girth to be maximally connected and super-connected, respectively. In view of complex mathematical expressions, we omit these results here.
However, for two special cases: connected graphs ($g\geq3$) and connected triangle-free graphs ($g\geq4$), we will provide improved and specific theorems in subsequent sections.
\end{remark}

\section{ Vertex-connectivity and $Q$-index of connected graphs ($g\geq3$)}

Throughout this section, we assume that $k$ and $\delta$ are positive integers.
The goal of this section is to investigate the relationship between
the connectivity and the $Q$-index of a graph.

\subsection{$k$-connected graphs ($g\geq3$)}

We will present a
lower bound on $q(G)$ for a connected graph to be $k$-connected.
Define $q_{0} = q(K_{k-1}\vee(K_{\delta-k+2}\cup K_{n-\delta-1}))$.
Direct computation yields that $q_0$ is the largest root of the equation

$
\lambda^{3}-(3n+k-7)\lambda^{2}+(2n^{2}-15n+3nk-4k+16+4(\delta-k+2)(n-\delta-1))\lambda-4(\delta-k+2)(n-\delta-1)(n-2)-
(n-k+1)(2n-k+1)(k-3)-(3n-2k+2)(k-3)^{2}-(k-3)^{3}=0.
$

\begin{theorem}\label{th4.1}
Let $G$ be a connected graph of order $n$ and minimum degree $\delta\geq k\geq2$.
Suppose that $q(G)\geq q_0$.
Then $G$ is $k$-connected if and only if $G \not \cong K_{k-1}\vee(K_{\delta-k+2}\cup K_{n-\delta-1})$.
\end{theorem}

\begin{proof} By definition, $K_{k-1}\vee(K_{\delta-k+2}\cup K_{n-\delta-1})$ has a $(k-1)$ vertex-cut and so
$\kappa(K_{k-1}\vee(K_{\delta-k+2}\cup K_{n-\delta-1}))= k-1$. Therefore, it suffices to prove
the sufficiency.

By contradiction, we assume that  $G \not \cong K_{k-1}\vee(K_{\delta-k+2}\cup K_{n-\delta-1})$
and $\kappa\leq k-1$. Let $C$ be a minimum vertex-cut of $G$, then $|C|=\kappa\leq k-1<\delta$.
Let $V_{0}, V_{1}, \ldots, V_{t-1}$~$(t\geq2)$ be the vertex sets of connected components of
$G-C$ with $|V_{0}|\leq|V_{1}|\leq \cdots \leq|V_{t-1}|$.

By Lemma \ref{le2.10} with $g\geq3$, for each $i$ with
$0\leq i\leq t-1$, we have
$$
|V_{i}|\geq \nu(\delta, g, \kappa)\geq\nu(\delta, 3, \kappa)=\delta-\kappa+1.
$$
Let $U=\bigcup_{i=1}^{t-1}V_{i}$. Then $\delta-\kappa+1\leq |V_{0}|\leq |U|\leq n-\delta-1$ and $|V_{0}|+|U|=n-\kappa$.
As $E(V_{0}, U)=\emptyset$, $G$ can be viewed as a subgraph of $K_{\kappa}\vee(K_{|V_{0}|}\cup K_{|U|})$,
and so
$$
q(G)\leq q(K_{\kappa}\vee(K_{|V_{0}|}\cup K_{|U|})).
$$

Let $G(\kappa,a,b)=K_{\kappa}\vee(K_{a}\cup K_{b})$, where $\delta-\kappa+1\leq a\leq b\leq n-\delta-1$ and $\kappa+a+b=n$.
Let $X=(x_{1}, x_{2}, \ldots, x_{n})^{T}$ be the Perron vector of $G$ corresponding to $q(G(\kappa,a,b))$.
Without loss of generality, let $x:=x_{i}, i\in K_{a}$; $y:=x_{j}, j\in K_{\kappa}$; $z:=x_{l}, l\in K_{b}$.
As $\lambda X=(D+A)X$, we have
\begin{equation*} \label{nu0}
\left\{
\begin{array}{ll}
\lambda x=(a-1+\kappa)x+(a-1)x+\kappa y,\\
\lambda y=ax+(n-1)y+(\kappa-1)y+b z,\\
\lambda z=\kappa y+(b-1+\kappa)z+(b-1)z.
\end{array}
\right.
\end{equation*}
It follows that  $q(K_{\kappa}\vee(K_{a}\cup K_{b}))$ is the largest root of the equation

$\lambda^{3}-(3n+\kappa-6)\lambda^{2}+(2n^{2}-12n+3n\kappa-4\kappa+12+4ab)\lambda-4ab(n-2)-
(n-\kappa)(2n-\kappa)(\kappa-2)-(3n-2\kappa)(\kappa-2)^{2}-(\kappa-2)^{3}=0$.

By algebraic manipulation, for $\lambda \geq n-2$, we have
\begin{equation} \label{1}
f(\lambda; \kappa, a, b)-f(\lambda; \kappa, \delta-\kappa+1, n-\delta-1)=4(\lambda-n+2)[ab-(\delta-\kappa+1)(n-\delta-1)]\geq0.
\end{equation}
By Lemma \ref{le2.2}, $q(G(\kappa,a,b))>n-2.$ Substituting $\lambda$ with $q(G(\kappa,a,b))$ in (\ref{1}), we have
$f(q(G(\kappa, a, b)); \kappa, \delta-\kappa+1, n-\delta-1)\leq0,$ and so
$$
q(G(\kappa, a, b))\leq q(G(\kappa, \delta-\kappa+1, n-\delta-1)).
$$
Therefore,
\begin{equation} \label{1a}
q(G)\leq q(K_{\kappa}\vee(K_{|V_{0}|}\cup K_{|U|}))\leq q(K_{\kappa}\vee(K_{\delta-\kappa+1}\cup K_{n-\delta-1})).
\end{equation}
Since $\kappa\leq k-1,$ we conclude that  $K_{\kappa}\vee(K_{\delta-\kappa+1}\cup K_{n-\delta-1})$ is a subgraph
of $K_{k-1}\vee(K_{\delta-k+2}\cup K_{n-\delta-1})$,  and
\begin{equation} \label{1b}
q(G)\leq q(K_{\kappa}\vee(K_{\delta-\kappa+1}\cup K_{n-\delta-1}))\leq q(K_{k-1}\vee(K_{\delta-k+2}\cup K_{n-\delta-1})).
\end{equation}
By the hypothesis of Theorem \ref{th4.1}, $q(G)\geq q(K_{k-1}\vee(K_{\delta-k+2}\cup K_{n-\delta-1}))$, and so we must have
$$
q(G)=q(K_{k-1}\vee(K_{\delta-k+2}\cup K_{n-\delta-1})).
$$
It follows that all the inequalities in (\ref{1a}) and (\ref{1b}) must be equalities.
Hence we must have $|V_{0}|=\delta-k+2$, $|U|=n-\delta-1$ and $\kappa=k-1.$ Therefore $G\cong K_{k-1}\vee(K_{\delta-k+2}\cup K_{n-\delta-1})$,
contrary to our assumption. This completes the proof of the theorem.
\hspace*{\fill}$\Box$
\end{proof}

Hong et al. obtained a sufficient condition on size $m$ for $k$-connected graphs, in
which the lower bound of the size is the special case when $g\geq3$ of Lemma \ref{le3.1}.

\begin{theorem} (Hong, Xia, Chen and Volkmann \cite{HXC})\label{le4.2}
Let $k\geq2$ be an integer. Let $G$ be a connected graph of order $n$, size $m$, and minimum degree $\delta\geq k.$ If $$m\geq\frac{1}{2}n(n-1)-(\delta-k+2)(n-\delta-1),$$ then
$G$ is $k$-connected unless $G\cong K_{k-1}\vee(K_{\delta-k+2}\cup K_{n-\delta-1})$.
\end{theorem}

Theorem \ref{le4.2} can be applied to show an explicit lower bound of $q(G)$ to predict $k$-connected graphs.

\begin{corollary}\label{th4.3}
Let $G$ be a connected graph with $n = |V(G)|$, $m = |E(G)|$ and $\delta = \delta(G) \geq k \ge 2$. Suppose that
$$
q(G)\geq 2(n-\delta+k-3)+\frac{2\delta(\delta-k+2)}{n-1}.
$$
Then $G$ is $k$-connected.
\end{corollary}

\begin{proof}
Suppose that $G$ is not $k$-connected. By assumption and Lemma \ref{le2.5}, we have
\begin{equation}\label{3n}
2(n-\delta+k-3)+\frac{2\delta(\delta-k+2)}{n-1}\leq q(G)\leq \frac{2m}{n-1}+n-2.
\end{equation}
Then $m\geq\frac{1}{2}n(n-1)-(\delta-k+2)(n-\delta-1).$
By Theorem \ref{le4.2}, $G\cong K_{k-1}\vee(K_{\delta-k+2}\cup K_{n-\delta-1}).$
Since $$|E(G)|=\frac{1}{2}n(n-1)-(\delta-k+2)(n-\delta-1),$$
the inequalities in (\ref{3n}) must be equalities. By Lemma \ref{le2.5}, $G\cong K_{n}$ or $K_{1, n-1}.$
As $K_{k-1}\vee(K_{\delta-k+2}\cup K_{n-\delta-1})$ is isomorphic to neither  $K_{n}$ nor $K_{1, n-1}$, a contradiction
is obtained. \hspace*{\fill}$\Box$
\end{proof}

Finally, we present a sufficient condition for a $k$-connected graph in terms of $q(\overline{G})$ to conclude this section.

\begin{theorem}\label{th4.4}
Let $a, \delta, k, n$ be positive integers satisfying $\delta-k+2\leq a\leq n-\delta-1$,
and $G$ be a connected graph of order $n$ and minimum degree $\delta \geq k \ge 2$. Suppose that
$$
q(\overline{G})\leq n-k+1.
$$
Then $G$ is $k$-connected if and only if $G \not\cong K_{k-1}\vee(K_{a}\cup K_{n-k+1-a})$.
\end{theorem}

\begin{proof}
By definition, $K_{k-1}\vee(K_{a}\cup K_{n-k+1-a})$ has a vertex-cut of cardinality $k-1$.
Thus we only need to prove the sufficiency of the theorem.
Suppose that $G \not\cong K_{k-1}\vee(K_{a}\cup K_{n-k+1-a})$ and $\kappa \leq k-1.$ Let $C$ is a minimum vertex-cut of $G$. Then
$|C|=\kappa\leq k-1<\delta$, and for some integer $t \ge 2$, $G - C$ has $t$ components.
Let $V_{0}, V_{1}, \ldots, V_{t-1}$ be the vertex sets of
connected components of $G-C$ satisfying $|V_{0}|\leq|V_{1}|\leq \cdots \leq|V_{t-1}|$.
By Lemma \ref{le2.10} with $g\geq3$, we have, for any $i$ with $0\leq i\leq t-1$,
$$
|V_{i}|\geq\nu(\delta, g, \kappa)\geq \nu(\delta, 3, \kappa)=\delta-\kappa+1.
$$
Let $U=\bigcup_{i=1}^{t-1}V_{i}.$ Then $\delta-\kappa+1\leq |V_{0}|\leq |U|\leq n-\delta-1$ and $|V_{0}|+|U|=n-\kappa$.
Since $E(V_{0}, U)=\emptyset,$ we conclude that  $K_{|V_{0}|, |U|}$ must be a subgraph of $\overline{G}$, and so
$$
q(\overline{G})\geq q(K_{|V_{0}|, |U|})=n-\kappa\geq n-k+1.
$$
It follows from the hypothesis of Theorem \ref{th4.4} that $q(\overline{G})=n-k+1$.
Hence we have $\kappa=k-1$ and $\overline{G}=K_{|V_{0}|, |U|}$. Let $|V_{0}|=a$.
Then $G\cong K_{k-1}\vee(K_{a}\cup K_{n-k+1-a}),$ where $\delta-k+2\leq a\leq n-\delta-1,$ contrary to our assumption.
\hspace*{\fill}$\Box$
\end{proof}

\subsection{Maximally connected graphs ($g\geq3$)}

We consider the problem how the $Q$-index of a graph warrants the property that
$G$ is maximally connected, that is, the condition $\kappa(G)=\delta(G)$ holds.
These can be obtained by taking $k=\delta$ in Theorem \ref{th4.1}, Corollary \ref{th4.3} and Theorem \ref{th4.4},
and so we have the following corollaries of the main results in Section 4.1.
Let $q_1 = q(K_{\delta-1}\vee(K_{2}\cup K_{n-\delta-1}))$, the
largest root of the equation

$
\lambda^{3}-(3n+\delta-7)\lambda^{2}+(2n^{2}-7n+3n\delta-12\delta+8)\lambda-8(n-\delta-1)(n-2)-(n-\delta+1)(\delta-3)(2n-\delta+1)-(3n-2\delta+2))(\delta-3)^{2}-(\delta-3)^{3}=0.
$

\begin{corollary}\label{th4.5}
Let $G$ be a connected graph with order $n$ and minimum degree $\delta \geq2$, and let $q(G)\geq q_1$.
Then $G$ is maximally connected if and only if $G \not\cong K_{\delta-1}\vee(K_{2}\cup K_{n-\delta-1}).$
\end{corollary}

\begin{corollary}\label{th4.6}
Let $G$ be a connected graph of order $n$ and minimum degree $\delta\geq2.$
If$$q(G)\geq 2(n-3)+\frac{4\delta}{n-1},$$
then $G$ is maximally connected.
\end{corollary}

\begin{corollary}\label{th4.7}
Let $G$ be a connected graph of order $n$ and minimum degree $\delta\geq2.$
If $$q(\overline{G})\leq n-\delta+1,$$
then $G$ is maximally connected if and only if $G\ncong K_{\delta-1}\vee(K_{a}\cup K_{n-\delta+1-a}),$ where $2\leq a\leq n-\delta-1.$
\end{corollary}

\subsection{Super-connected graphs ($g\geq3$)}

Let $q_2 = q(K_{\delta}\vee(K_{2}\cup K_{n-\delta-2}))$. By definition, $q_2$ is
the largest root of the equation

$
\lambda^{3}-(3n+\delta-6)\lambda^{2}+(2n^{2}-4n+3n\delta-12\delta-4)\lambda-8(n-\delta-2)(n-2)-(n-\delta)(2n-\delta)(\delta-2)-(3n-2\delta)(\delta-2)^{2}-(\delta-2)^{3}=0.
$
\begin{theorem}\label{th4.8}
Let $G$ be a connected graph of order $n$ and minimum degree $\delta$. If $q(G)\geq q_2$,
then $G$ is super-$\kappa$.
\end{theorem}
\begin{proof}
By contradiction, we assume that $G$ is not super-$\kappa$.
Then $G$ contains a minimum vertex-cut with $|C|=\kappa\leq \delta$. Therefore, for some integer $t \ge 2$,
$G - C$ has $t$ components, whose vertex sets are respectively denoted by
$V_{0}, V_{1}, \ldots, V_{t-1}$, such that $2\leq|V_{0}|\leq|V_{1}|\leq \cdots \leq|V_{t-1}|$.
Let $U=\bigcup_{i=1}^{t-1}V_{i}$. Then we have $2\leq |V_{0}|\leq |U|\leq n-\kappa-2$ and $|V_{0}|+|U|=n-\kappa$.
As $E(V_{0}, U)=\emptyset$, it follows that $G$ is a subgraph of $K_{\kappa}\vee(K_{|V_{0}|}\cup K_{|U|})$, and so
$$
q(G)\leq q(K_{\kappa}\vee(K_{|V_{0}|}\cup K_{|U|})).
$$
With an argument  similar to that of Theorem \ref{th4.1}, we conclude that
$q(K_{\kappa}\vee(K_{|V_{0}|}\cup K_{|U|}))$ is the largest root of the equation

$
\lambda^{3}-(3n+\kappa-6)\lambda^{2}+(2n^{2}-12n+3n\kappa-4\kappa+12+4|V_{0}|\cdot|U|)\lambda-4|V_{0}|\cdot|U|(n-2)-(n-\kappa)(2n-\kappa)(\kappa-2)-(3n-2\kappa)(\kappa-2)^{2}-(\kappa-2)^{3}=0.
$

Direct computation yields that, if $\lambda\geq n-2$, then
\begin{equation}\label{2}
f(\lambda; \kappa, |V_{0}|, |U|)-f(\lambda; \kappa, 2, n-\kappa-2)=4(\lambda-n+2)(|V_{0}|\cdot|U|-2(n-\kappa-2))\geq0.
\end{equation}
By Lemma \ref{le2.2}, $q(K_{\kappa}\vee(K_{|V_{0}|}\cup K_{|U|}))\geq n-2$.
Substituting $\lambda$ with $q(K_{\kappa}\vee(K_{|V_{0}|}\cup K_{|U|}))$ in (\ref{2}), we have
$f(q(K_{\kappa}\vee(K_{|V_{0}|}\cup K_{|U|})); \kappa, 2, n-\kappa-2)\leq0$. It follows that
$$
q(K_{\kappa}\vee(K_{|V_{0}|}\cup K_{|U|}))\leq q(K_{\kappa}\vee(K_{2}\cup K_{n-\kappa-2})),
$$
which implies that
\begin{equation}\label{2a}
q(G)\leq q(K_{\kappa}\vee(K_{|V_{0}|}\cup K_{|U|}))\leq q(K_{\kappa}\vee(K_{2}\cup K_{n-\kappa-2})).
\end{equation}
Since $\kappa\leq\delta,$ we observe that $K_{\kappa}\vee(K_{2}\cup K_{n-\kappa-2})$
is a subgraph of $K_{\delta}\vee(K_{2}\cup K_{n-\delta-2})$, and so
\begin{equation}\label{2a1}
q(G)\leq q(K_{\kappa}\vee(K_{2}\cup K_{n-\kappa-2}))\leq q(K_{\delta}\vee(K_{2}\cup K_{n-\delta-2})).
\end{equation}
By the assumption of Theorem \ref{th4.8}, we have $q(G)=q(K_{\delta}\vee(K_{2}\cup K_{n-\delta-2}))$.
Thus the inequalities in (\ref{2a}) and (\ref{2a1}) must be equalities. It follows that $\kappa=\delta$, $|V_{0}|=2$ and $|U|=n-\delta-2,$
and so $G\cong K_{\delta}\vee(K_{2}\cup K_{n-\delta-2})$. However, $\delta(G)=\delta+1>\delta,$
contrary to the choice of $G$. This justifies the theorem.
\hspace*{\fill}$\Box$
\end{proof}

Hong et al. obtained the following sufficient condition on size $m$ for super-connected graphs.
This again, can be applied to obtain a relationship between the $Q$-index and super-$\kappa$ property of
a connected graph $G$.

\begin{theorem}(Hong, Xia, Chen and Volkmann \cite{HXC})\label{le4.9}
Let $G$ be a connected graph of order $n$, size $m$ and minimum degree $\delta.$ If $$m\geq\frac{1}{2}(n-2)(n-3)+2\delta,$$ then
$G$ is super-$\kappa$ unless $G\cong(K_{\delta}\vee(K_{2}\cup K_{n-\delta-2}))-e$, where $e=xy$ is an edge of $K_{\delta}\vee(K_{2}\cup K_{n-\delta-2})$, and $d(x)=\delta+1,$ $d(y)=n-1$.
\end{theorem}

\begin{corollary}\label{th4.10}
Let $G$ be a connected graph with $n$ vertices and minimum degree $\delta.$
If $$
q(G)\geq2(n-3)+\frac{4\delta+2}{n-1},
$$
then $G$ is super-$\kappa$.
\end{corollary}

\begin{proof}
Suppose that $G$ is not super-$\kappa$. By assumption and Lemma \ref{le2.5},
\begin{equation}\label{2b}
2(n-3)+\frac{4\delta+2}{n-1}\leq q(G)\leq \frac{2m}{n-1}+n-2.
\end{equation}
Then we have $m\geq\frac{1}{2}(n-2)(n-3)+2\delta.$ By Theorem \ref{le4.9}, $G\cong(K_{\delta}\vee(K_{2}\cup K_{n-\delta-2}))-e$,
where $e=xy$ is an edge of $K_{\delta}\vee(K_{2}\cup K_{n-\delta-2})$ with $d(x)=\delta+1,$ $d(y)=n-1$.
Since
$$
|E(G)|=\frac{n(n-1)}{2}-2(n-\delta-2)-1=\frac{(n-2)(n-3)}{2}+2\delta,
$$
the inequalities in (\ref{2b}) should be equalities. By Lemma \ref{le2.5}, $G\cong K_{n}$ or $K_{1, n-1}$.
As $(K_{\delta}\vee(K_{2}\cup K_{n-\delta-2}))-e$ is isomorphic to neither  $K_{n}$ nor $K_{1, n-1}$, a contradiction
is obtained. Thus $G$ must be super-$\kappa$. \hspace*{\fill}$\Box$
\end{proof}

Finally, we present a sufficient condition
for a super-connected graph in terms of $q(\overline{G})$ to conclude the section.

\begin{theorem}\label{th4.11}
Let $G$ be a connected graph with $n$ vertices and minimum degree $\delta.$ If $$q(\overline{G})\leq n-\delta,$$
then $G$ is super-$\kappa$.
\end{theorem}

\begin{proof}
Suppose that $G$ is not super-$\kappa$. Then $G$ has a  minimum vertex-cut $C$ with
$|C|=\kappa\leq \delta$ such that for some integer $t \ge 2$, $G-C$ has $t$ components.
Let $V_{0}, V_{1}, \ldots, V_{t-1}$ be the vertex sets of connected components of
$G-C$ with $|V_{0}|\leq|V_{1}|\leq \cdots \leq|V_{t-1}|$.
Let $U=\bigcup_{i=1}^{t-1}V_{i}.$ Then $2\leq |V_{0}|\leq |U|\leq n-\kappa-2$ and $|V_{0}|+|U|=n-\kappa$.
As $E(V_{0}, U)=\emptyset$, we conclude that $K_{|V_{0}|, |U|}$ is a subgraph of $\overline{G}$, and so
$$
q(\overline{G})\geq q(K_{|V_{0}|, |U|})=n-\kappa \geq n-\delta.
$$
By assumption, $q(\overline{G})\leq n-\delta$, and so we have $q(\overline{G})=n-\delta$ and $\overline{G}\cong K_{|V_{0}|, |U|}$.
Thus we must have $\kappa=\delta$ and $G\cong K_{\delta}\vee(K_{|V_{0}|}\cup K_{|U|})$.
Since $\delta(K_{\delta}\vee(K_{|V_{0}|}\cup K_{|U|}))\geq\delta+1>\delta$, contrary to the assumption
on the choice of $G$.
\hspace*{\fill}$\Box$
\end{proof}

\section{ Vertex-connectivity and $Q$-index of triangle-free graphs ($g\geq4$)}

\subsection{$k$-connected triangle-free graphs ($g\geq4$)}

Hong et al. obtained a sufficient condition on size $m$ to warrant $k$-connected graphs,
in which the lower  bound on the graph size is a special case of Lemma \ref{le3.1} when $g\geq4$.
This can, once again, be applied to obtain results relating the $Q$-index and the connectivity
in a connected triangle-free graph.

\begin{theorem} (Hong, Xia, Chen and Volkmann \cite{HXC})\label{le5.1}
Let $G$ be a connected triangle-free graph of order $n$, size $m$ and minimum degree
$\delta\geq k\geq2.$ If $$m\geq\delta^{2}+\lfloor\frac{1}{4}(n-2\delta+k-1)^{2}\rfloor,$$ then
$G$ is $k$-connected unless $V(G)=X\cup C\cup Y$, and $C$ is
a minimum vertex-cut of $G$ with $G[C]\cong\overline{K_{k-1}}$, $G[X\cup C]\cong K_{\delta, \delta}$ and
$G[Y\cup C]\cong K_{\lfloor\frac{n-2\delta+k-1}{2}\rfloor, \lceil\frac{n-2\delta+k-1}{2}\rceil}$.
\end{theorem}

\begin{corollary}\label{th5.2}
Let $G$ be a connected triangle-free graph of order $n$ and minimum degree $\delta\geq k\geq2.$ If
$$
q(G)\geq n+k-2\delta-2+\frac{2\delta^{2}}{n-1}+\lfloor\frac{1}{2}(n-1+\frac{(k-2\delta)^{2}}{n-1})\rfloor,
$$
then
$G$ is $k$-connected.
\end{corollary}

\begin{proof}
Suppose that $G$ is not $k$-connected. By assumption and Lemma \ref{le2.5}, we have
\begin{equation}\label{3m}
n+k-2\delta-2+\frac{2\delta^{2}}{n-1}+\lfloor\frac{1}{2}(n-1+\frac{(k-2\delta)^{2}}{n-1})\rfloor \leq q(G)\leq\frac{2m}{n-1}+n-2.
\end{equation}
Then $m\geq\delta^{2}+\lfloor\frac{1}{4}(n-2\delta+k-1)^{2}\rfloor.$
By Theorem \ref{le5.1}, $V(G)=X\cup C\cup Y$, and $C$ is
a minimum vertex-cut of $G$ with $G[C]\cong\overline{K_{k-1}}$, $G[X\cup C]\cong K_{\delta, \delta}$ and
$G[Y\cup C]\cong K_{\lfloor\frac{n-2\delta+k-1}{2}\rfloor, \lceil\frac{n-2\delta+k-1}{2}\rceil}.$
Since $$|E(G)|=\delta^{2}+\lfloor\frac{1}{4}(n-2\delta+k-1)^{2}\rfloor,$$
the inequalities in (\ref{3m}) must be equalities. By Lemma \ref{le2.5}, $G\cong K_{n}$ or $K_{1, n-1}.$
However, $G$ is isomorphic to neither $K_{n}$ nor $K_{1, n-1}$, a contradiction.
\hspace*{\fill}$\Box$
\end{proof}

Finally, we present a sufficient condition for a $k$-connected triangle-free graph in terms of $q(\overline{G})$ to conclude the section.
\begin{theorem}\label{th5.3}
Let $G$ be a connected triangle-free graph of order $n$ and minimum degree $\delta\geq k\geq2,$ and $\overline{G}$ be connected.
If
\begin{equation} \label{3aa}
q(\overline{G})\leq 2(n-1)-\frac{4\delta^{2}}{n}-\lfloor\frac{(n-2\delta+k-1)^{2}}{n}\rfloor,
\end{equation}
then $G$ is $k$-connected.
\end{theorem}

\begin{proof}
By contradiction, assume that $G$ has a minimum vertex-cut $C$ with $|C|=\kappa(G) \leq k-1<\delta$.
Let $V_{0}, V_{1}, \ldots, V_{t-1}$, for some integer $t \ge 2$, be the vertex sets of connected components of
$G-C$ satisfying $|V_{0}|\leq|V_{1}|\leq \cdots \leq|V_{t-1}|$.
By Lemma \ref{le2.10} with $g\geq4$, we have, for any $i$ with $0\leq i\leq t-1$,
$$
|V_{i}|\geq \nu(\delta, g, \kappa)\geq\nu(\delta, 4, \kappa)=2\delta-\kappa.
$$
Let $U=\bigcup_{i=1}^{t-1}V_{i}$. Then $2\delta-\kappa\leq |V_{0}|\leq |U|\leq n-2\delta$ and $|V_{0}|+|U|=n-\kappa$.

By Theorem \ref{le2.13}, with similar analysis of Theorem 5.2 in \cite{HXC}, we have
\begin{eqnarray*}
  m(G) &=& |E(G[V_{0}\cup C])|+|E(G[U\cup C])|-|E(G[C])|\\
      &\leq& \lfloor\frac{(|V_{0}|+|C|)^{2}}{4}\rfloor+\lfloor\frac{(|U|+|C|)^{2}}{4}\rfloor-|E(G[C])|\\
      &\leq& \lfloor\frac{(|V_{0}|+|C|)^{2}}{4}\rfloor+\lfloor\frac{(|U|+|C|)^{2}}{4}\rfloor\\
       &=& \lfloor\frac{(|V_{0}|+|U|+|C|)^{2}+|C|^{2}}{4}-\frac{|V_{0}|\cdot|U|}{2}\rfloor\\
       &=& \lfloor\frac{n^{2}+\kappa^{2}}{4}-\frac{|V_{0}|\cdot|U|}{2}\rfloor\leq \lfloor\frac{n^{2}+\kappa^{2}}{4}-\frac{(2\delta-\kappa)\cdot(n-2\delta)}{2}\rfloor\\
       &=& \delta^{2}+\lfloor\frac{(n-2\delta+\kappa)^{2}}{4}\rfloor
      \leq \delta^{2}+\lfloor\frac{(n-2\delta+k-1)^{2}}{4}\rfloor.
\end{eqnarray*}
Since $m(G)+m(\overline{G})=\frac{n(n-1)}{2},$ then
\begin{equation} \label{3a}
m(\overline{G}) = \frac{n(n-1)}{2}-m(G)
              \ge  \frac{n(n-1)}{2}-\delta^{2}-\lfloor\frac{(n-2\delta+k-1)^{2}}{4}\rfloor.
\end{equation}
By Lemma \ref{le2.1},
\begin{equation} \label{3b}
q(\overline{G}) \ge  \frac{4m(\overline{G})}{n}
              \ge \frac{4}{n}[\frac{n(n-1)}{2}-\delta^{2}-\lfloor\frac{(n-2\delta+k-1)^{2}}{4}\rfloor]
              = 2(n-1)-\frac{4\delta^{2}}{n}-\lfloor\frac{(n-2\delta+k-1)^{2}}{n}\rfloor.
\end{equation}
Combine (\ref{3aa}) and (\ref{3b}) to get
$q(\overline{G})=2(n-1)-\frac{4\delta^{2}}{n}-\lfloor\frac{(n-2\delta+k-1)^{2}}{n}\rfloor.$
Then all the inequalities in (\ref{3a}) and (\ref{3b}) must be equalities.
It follows that $|C|=\kappa=k-1,$ $|V_{0}|=2\delta-k+1,$ $|U|=n-2\delta,$ $|E(G[C])|=0,$
$|E(G[V_{0}\cup C])|=\delta^{2},$  $|E(G[U\cup C])|=\lfloor\frac{(n-2\delta+k-1)^{2}}{4}\rfloor$, $G[V_{0}\cup C]=K_{\lfloor \frac{(|V_{0}|+|C|)}{2}\rfloor, \lceil \frac{(|V_{0}|+|C|)}{2}\rceil}$, $G[U\cup C]=K_{\lfloor \frac{(|U|+|C|)}{2}\rfloor, \lceil \frac{(|U|+|C|)}{2}\rceil}$ and $\overline{G}$ is regular.
Therefore, $G[C]=\overline{K_{k-1}}$, $G[V_{0}\cup C]=K_{\delta, \delta}$ and $G[U\cup C]=K_{\lfloor\frac{n-2\delta+k-1}{2}\rfloor, \lceil\frac{n-2\delta+k-1}{2}\rceil}$. However, $G$ is not regular, and so $\overline{G}$ cannot be regular, a contradiction.
\hspace*{\fill}$\Box$
\end{proof}

\subsection{Maximally connected triangle-free graphs ($g\geq4$)}
Naturally, by setting $k=\delta$ in Corollary \ref{th5.2} and Theorem \ref{th5.3},
we can obtain the following results on maximally connected triangle-free graphs.

\begin{corollary}\label{th5.4}
Let $G$ be a connected triangle-free graph of order $n$ and minimum degree $\delta\geq2.$ If $$q(G)\geq n-\delta-2+\frac{2\delta^{2}}{n-1}+\lfloor\frac{1}{2}(n-1+\frac{\delta^{2}}{n-1})\rfloor,$$ then
$G$ is maximally connected.
\end{corollary}

\begin{corollary}\label{th5.5}
Let $G$ be a connected triangle-free graph of order $n$ and minimum degree $\delta\geq2,$ and $\overline{G}$ be connected.
If$$q(\overline{G})\leq 2(n-1)-\frac{4\delta^{2}}{n}-\lfloor\frac{(n-\delta-1)^{2}}{n}\rfloor,$$
then $G$ is maximally connected.
\end{corollary}

\subsection{Super-connected triangle-free graphs ($g\geq4$)}

We start quoting a theorem by Hong et al. \cite{HXC} again, to be applied in one of our results.

\begin{theorem} (Hong, Xia, Chen and Volkmann \cite{HXC})\label{le5.6}
Let $G$ be a connected triangle-free graph of order $n$, size $m$ and minimum degree $\delta\geq 2.$ If $$m\geq\delta^{2}+\lfloor\frac{1}{4}(n-\delta)^{2}\rfloor,$$ then
$G$ is super-$\kappa$.
\end{theorem}

\begin{corollary}\label{th5.6}
Let $G$ be a connected triangle-free graph of order $n$ and minimum degree $\delta\geq 2.$  If $$q(G)\geq n-\delta-1+\frac{2\delta^{2}}{n-1}+\lfloor\frac{1}{2}(n-1+\frac{(\delta-1)^{2}}{n-1})\rfloor,$$ then
$G$ is super-$\kappa$.
\end{corollary}

\begin{proof}
By assumption and Lemma \ref{le2.5}, we have
$$n-\delta-1+\frac{2\delta^{2}}{n-1}+\lfloor\frac{1}{2}(n-1+\frac{(\delta-1)^{2}}{n-1})\rfloor \leq q(G) \leq \frac{2m}{n-1}+n-2,$$
and so $m\geq\delta^{2}+\lfloor\frac{1}{4}(n-\delta)^{2}\rfloor$. By Theorem \ref{le5.6}, $G$ is super-$\kappa$.
\hspace*{\fill}$\Box$
\end{proof}

\begin{theorem}\label{th5.7}
Let $G$ be a connected triangle-free graph of order $n$ and minimum degree $\delta\geq 2,$ and $\overline{G}$ be connected.
If
\begin{equation}\label{3bb}
q(\overline{G})\leq 2(n-1)-\frac{4\delta^{2}}{n}-\lfloor\frac{(n-\delta)^{2}}{n}\rfloor,
\end{equation}
then $G$ is super-$\kappa$.
\end{theorem}

\begin{proof}
Suppose that $G$ is not super-$\kappa$. Note that
\begin{eqnarray*}
q(\overline{G})&\leq& 2(n-1)-\frac{4\delta^{2}}{n}-\lfloor\frac{(n-\delta)^{2}}{n}\rfloor\\
               &\leq& 2(n-1)-\frac{4\delta^{2}}{n}-\lfloor\frac{(n-\delta-1)^{2}}{n}\rfloor,
\end{eqnarray*}
by Corollary \ref{th5.5}, we have $\kappa=\delta.$

Let $C$ be the minimum vertex-cut with $|C|=\kappa=\delta.$ Let $V_{0}, V_{1}, \ldots, V_{t-1}$~$(t\geq2)$ be the vertex sets of connected components of $G-C$ with $2\leq|V_{0}|\leq|V_{1}|\leq \cdots \leq|V_{t-1}|$. Let $U=\bigcup_{i=1}^{t-1}V_{i}.$ Then $|V_{0}|+|U|=n-\kappa$. By Lemma \ref{le2.10} with $g\geq4,$ then $|V_{0}|\geq \nu(\delta, g, \kappa)\geq\nu(\delta, 4, \kappa)=2\delta-\kappa.$
So we have
$$2\leq\delta=2\delta-\kappa \leq |V_{0}| \leq |U| \leq n-2\delta.$$
By Theorem \ref{le2.13}, we have
\begin{eqnarray*}
  m(G) &=& |E(G[V_{0}\cup C])|+|E(G[U\cup C])|-|E(G[C])|\\
      &\leq& \lfloor\frac{(|V_{0}|+|C|)^{2}}{4}\rfloor+\lfloor\frac{(|U|+|C|)^{2}}{4}\rfloor-|E(G[C])|\\
      &\leq& \lfloor\frac{(|V_{0}|+|C|)^{2}}{4}\rfloor+\lfloor\frac{(|U|+|C|)^{2}}{4}\rfloor\\
       &=& \lfloor\frac{(|V_{0}|+|U|+|C|)^{2}+|C|^{2}}{4}-\frac{|V_{0}|\cdot|U|}{2}\rfloor\\
       &=& \lfloor\frac{n^{2}+\kappa^{2}}{4}-\frac{|V_{0}|\cdot|U|}{2}\rfloor
\leq \lfloor\frac{n^{2}+\kappa^{2}}{4}-\frac{\delta \cdot(n-2\delta)}{2}\rfloor\\
       &=& \delta^{2}+\lfloor\frac{(n-\delta)^{2}}{4}\rfloor
\end{eqnarray*}
Since $m(G)+m(\overline{G})=\frac{n(n-1)}{2},$ then
\begin{equation}\label{3c}
m(\overline{G})= \frac{n(n-1)}{2}-m(G)\geq \frac{n(n-1)}{2}-\delta^{2}-\lfloor\frac{(n-\delta)^{2}}{4}\rfloor.
\end{equation}
By Lemma \ref{le2.1},
\begin{equation}\label{3d}
q(\overline{G})\geq\frac{4m(\overline{G})}{n}\geq \frac{4}{n}[\frac{n(n-1)}{2}-\delta^{2}-\lfloor\frac{(n-\delta)^{2}}{4}\rfloor]
                = 2(n-1)-\frac{4\delta^{2}}{n}-\lfloor\frac{(n-\delta)^{2}}{n}\rfloor.
\end{equation}
Combine (\ref{3bb}) and (\ref{3d}) to get $q(\overline{G})=2(n-1)-\frac{4\delta^{2}}{n}-\lfloor\frac{(n-\delta)^{2}}{n}\rfloor.$
Then all the inequalities in (\ref{3c}) and (\ref{3d}) must be equalities. It follows that $|C|=\delta,$ $|V_{0}|=\delta,$ $|U|=n-2\delta,$ $|E(G[C])|=0,$
$|E(G[V_{0}\cup C])|=\delta^{2},$  $|E(G[U\cup C])|=\lfloor\frac{(n-\delta)^{2}}{4}\rfloor$, $G[V_{0}\cup C]=K_{\lfloor \frac{(|V_{0}|+|C|)}{2}\rfloor, \lceil \frac{(|V_{0}|+|C|)}{2}\rceil}$, $G[U\cup C]=K_{\lfloor \frac{(|U|+|C|)}{2}\rfloor, \lceil \frac{(|U|+|C|)}{2}\rceil}$ and $\overline{G}$ is regular.
Therefore, $G[C]=\overline{K_{\delta}}$, $G[V_{0}\cup C]=K_{\delta, \delta}$ and $G[U\cup C]=K_{\lfloor\frac{n-\delta}{2}\rfloor, \lceil\frac{n-\delta}{2}\rceil}$. Howerver, $G$ is not regular. So $\overline{G}$ cannot be regular, a contradiction.
\hspace*{\fill}$\Box$
\end{proof}

\noindent
{\bf Acknowledgement.} The research of Huicai Jia is supported by NSFC (No.~11701148) and Natural Science Foundation of Education Ministry of Henan Province (18B110005). The research of Hong-Jian Lai is supported by NSFC (Nos.~11771039 and 11771443).
The research of Ruifang Liu is supported by Outstanding Young Talent Research Fund of Zhengzhou University (No.~1521315002), China Postdoctoral Science Foundation (No.~2017M612410) and Foundation for University Key Teacher of Henan Province (No.~2016GGJS-007).

\end{document}